\documentclass[11pt]{article}

\usepackage{amsmath,amsthm,amssymb}
\usepackage{pstricks,pst-grad,pst-node,pst-tree}

\theoremstyle{plain}
\newtheorem{theorem}{Theorem}

\newtheorem{corollary}[theorem]{Corollary}

\title{Threshold Digraphs\thanks{Official contribution of the National
Institute of Standards and Technology; not subject to copyright in the
United States.}}

\author{Brian~Cloteaux\footnotemark[2] \and
        M.~Drew~LaMar\footnotemark[3] \and
	Elizabeth~Moseman\footnotemark[2]\ \footnotemark[4] \and
	James~Shook\footnotemark[2]
	}

\date {}
\begin{document}

\maketitle

\renewcommand{\thefootnote}{\fnsymbol{footnote}}
\footnotetext[2]{
National Institute of Standards and Technology,
Applied and Computational Mathematics Division,
Gaithersburg, MD
\texttt{\{brian.cloteaux, elizabeth.moseman, james.shook\}@nist.gov} }
\footnotetext[3]{
The College of William and Mary, Department of Biology,
Williamsburg, VA, \texttt{mdlama@wm.edu} }
\footnotetext[4]{
Corresponding author. }

\renewcommand{\thefootnote}{\arabic{footnote}}

\begin{abstract}
  A digraph whose degree sequence has a unique vertex labeled
  realization is called threshold.  In this paper we present several
  characterizations of threshold digraphs and their degree sequences,
  and show these characterizations to be equivalent. One of the
  characterizations is new, and allows for a shorter proof of the
  equivalence of the two known characterizations as well as proving
  the final characterization which appears without proof in the
  literature.  Using this result, we obtain a new, short proof of the
  Fulkerson-Chen theorem on degree sequences of general digraphs.
\end{abstract}

\section{Introduction}

What follows is a brief introduction to the notation used in the
paper. For notation not otherwise defined, see
Diestel~\cite{Diestel2010}.  We let $G= (V,E)$ be a digraph where $E$
is a set of ordered pairs called arcs.  If $(v,w)\in E$, then we say
$w$ is an out-neighbor of $v$ and $v$ is an in-neighbor of $w$. We
notate the out-degree of a vertex $v\in V$ by $d^+_G(v)$ and the
in-degree as $d^-_G(v)$, suppressing the subscript when the underlying
digraph is apparent from context.

Given a sequence $\alpha =
\big((\alpha_1^+,\alpha_1^-),\ldots,(\alpha_n^+,\alpha_n^-)\big)$ of
integer pairs we say that $\alpha$ is {\bf digraphical} if there is a
digraph $G = (V, E)$ with $V = \{ v_1, \ldots, v_n\}$ and $d^+(v_i) =
\alpha_i^+$, $d^-(v_i) = \alpha_i^-$. We call such $G$ a
\textbf{realization} of $\alpha$. An integer pair sequence $\alpha$ is
in {\bf positive lexicographical order} if $\alpha_i^+ \ge
\alpha_{i+1}^+$ with $\alpha_i^- \ge \alpha_{i+1}^-$ when $\alpha_i^+
= \alpha_{i+1}^+$.

We are interested in the degree sequences that have unique
vertex labeled realizations and the digraphs that realize them. As seen in~\cite{Mahadev1995}, the
undirected degree sequences with unique realizations are well studied.
Theorem~\ref{thm:td} in Section~\ref{S: thresholdChar} presents
several characterizations of this type of degree sequence and its
realization. We then show these characterizations to be
equivalent. One of the characterizations is new, and allows for a
shorter proof of the equivalence of the two known characterizations as
well as proving the final characterization which appears without proof
in the literature.  In Section~\ref{S: DigraphRealiz}, we use
Theorem~\ref{thm:td} to obtain a new short proof of the Fulkerson-Chen
theorem on degree sequences of general digraphs.  We end by presenting
some applications in Section~\ref{S: Applications}.

\begin{figure}
\begin{centering}
\begin{pspicture}(-3,-0.5)(7,2.5)
\psset{radius=0.1,arrowsize=1.5pt 5,arcangle=-15}
\Cnode*(0,0){ll}\nput*{-135}{ll}{$w$}
\Cnode*(0,2){ul}\nput*{135}{ul}{$y$}
\Cnode*(2,0){lr}\nput*{-45}{lr}{$x$}
\Cnode*(2,2){ur}\nput*{45}{ur}{$z$}
\ncline{->}{ll}{lr}
\ncline{->}{ul}{ur}
\ncline[linestyle=dashed]{->}{ll}{ur}
\ncline[linestyle=dashed]{->}{ul}{lr}

\Cnode*(5,0){tl}\nput*{-135}{tl}{$x$}
\Cnode*(7,0){tr}\nput*{-45}{tr}{$y$}
\Cnode*(6,1.7){tt}\nput*{90}{tt}{$z$}
\ncarc{->}{tl}{tr}
\ncarc{->}{tr}{tt}
\ncarc{->}{tt}{tl}
\ncarc[linestyle=dashed]{->}{tr}{tl}
\ncarc[linestyle=dashed]{->}{tl}{tt}
\ncarc[linestyle=dashed]{->}{tt}{tr}
\end{pspicture}
\end{centering}
\caption{A 2-switch (left) and an induced directed 3-cycle
  (right). Solid arcs must appear in the digraph and dashed arcs must
  not appear in the digraph. If an arc is not listed, then it may or may not
  be present.}\label{fig:forbidden}
\end{figure}
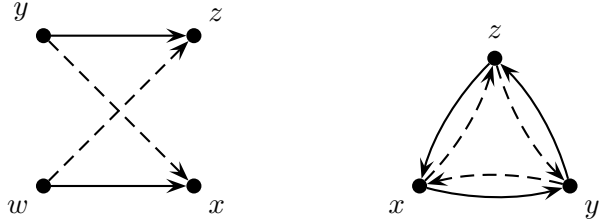
\section{Threshold Digraph Characterization}\label{S: thresholdChar}

In the existing literature~\cite{Rao1996}, the characterization of the unique
realization of a degree sequence is in terms of forbidden
configurations. The two forbidden configurations are the 2-switch and
the induced directed 3-cycle.
A {\bf 2-switch} is a set of four vertices $w, x, y, z$ so that
$(w,x)$ and $(y,z)$ are arcs of $G$ and $(w,z)$, $(y,x)$ are not. An
{\bf induced directed 3-cycle} is a set of three vertices $x,y,z$ so
that $(x,y),(y,z), (z,x)$ are arcs but there are no other arcs among
the vertices. Replacement of the arcs in these configurations with the
arcs that are not present yields another digraph with the same
degrees, both in and out, so any degree sequence of a digraph with
these configurations has multiple realizations. These configurations
are pictured in Figure~\ref{fig:forbidden}.

Our main theorem shortens the existing proofs by showing the
equivalence of our characterization
(Theorem~\ref{thm:td}.\ref{it:adj}) to known characterizations.

\begin{theorem}\label{thm:td} Let $G$ be a digraph and $A= [a_{ij}]$
an adjacency matrix of $G$. Define $\alpha_i^+ = \sum_{j=1}^n a_{ij}$
and $\alpha_j^-= \sum_{i=1}^n a_{ij}$. Suppose that the vertices
$v_1,\ldots, v_n$ of $G$ are ordered so that $d^+(v_i) = \alpha_i^+$,
$d^-(v_i) = \alpha_i^-$ and the degree sequence $\alpha =
\big((\alpha_1^+,\alpha_1^-),\ldots,(\alpha_n^+,\alpha_n^-)\big)$ of
$G$ is in positive lexicographic order. The following are equivalent:

\begin{enumerate}
\item $G$ is the unique labeled realization of the degree sequence
$\alpha$.\label{it:def}
\item There are no 2-switches or induced directed 3-cycles in
$G$.\label{it:td}
\item For every triple of distinct indices $i$, $j$ and $k$ with $i <
j$, if $a_{jk} =1$, then $a_{ik} = 1$.\label{it:adj}
\item The Fulkerson-Chen inequalities are satisfied with equality. In
other words, for $1 \le k \le n$,\label{it:fulk}
\[ \sum_{i=1}^k\min(\alpha_i^-,k-1)+\sum_{i=k+1}^n\min(\alpha_i^-,k) =
\sum_{i=1}^k\alpha_i^+.
\]
\end{enumerate}
\end{theorem}

\begin{proof} The equivalence of $(\ref{it:def})$ and $(\ref{it:td})$
has been shown previously in \cite{Rao1996}. For this, we need only
the implication $(\ref{it:def})\Rightarrow(\ref{it:td})$, which is
shown by the contrapositive: if there were a 2-switch or an induced
directed 3-cycle in $G$, then we can form another graph $G'$ on the
same degree sequence so $G$ does not have a unique realization. Notice
that this implication does not require positive lexicographic order.

$(\ref{it:td})\Rightarrow (\ref{it:adj})$ (Proof by contrapositive:
$\neg(\ref{it:adj})\Rightarrow\neg(\ref{it:td})$.)  Let $n\ge 3$ and
$i$, $j$, $k$ distinct indices so that $i < j$, $a_{jk} =1$ and
$a_{ik} = 0$. Let $l \notin\{i,j,k\}$, if such an index exists, and
note that what follows holds vacuously if $n=3$ and no such $l$
exists.  For this $l$, if $a_{il} = 1$ and $a_{jl} = 0$, then the arcs
$(v_i,v_l)$ and $(v_j,v_k)$ form a 2-switch.  Otherwise, define
$\kappa(x,y) = |\{l \notin\{i,j,k\}\mid a_{il} = x, a_{jl} = y\}|$ for
$x,y\in \{0,1\}$ and notice that $\kappa(1,0) = 0$. Thus, $\alpha_i^+
= a_{ij} + \kappa(1,1)$ and $\alpha_j^+ = a_{ji} + 1 + \kappa(1,1) +
\kappa(0,1)$. Since $\alpha_i^+ \ge \alpha_j^+$, we have $a_{ij} \ge
a_{ji} + 1 + \kappa(0,1)$ so $a_{ij} = 1$, $a_{ji} = 0$, $\kappa(0,1)
= 0$ and $\alpha_i^+ = \alpha_j^+$.

Now we consider the in-degree of $v_i$ and $v_j$. Since $a_{ij} = 1$,
$a_{ji} = 0$ and $\alpha_i^- \ge \alpha_j^-$ there must be a vertex
$v$ so that $(v,v_i)$ is an arc and $(v,v_j)$ is not an arc. If $v =
v_k$, then the vertices $v_i$, $v_j$ and $v_k$ form an induced
directed 3-cycle. Otherwise, set $v = v_l$
and consider $a_{lk}$. If $a_{lk} = 0$, then the arcs $(v_l, v_i)$ and
$(v_j,v_k)$ form a 2-switch. Otherwise, $a_{lk} = 1$ and the arcs
$(v_l,v_k)$ and $(v_i,v_j)$ form a 2-switch. 

$(\ref{it:adj})\Rightarrow (\ref{it:fulk})$
Let $A_k$ be the $k\times n$ submatrix of $A$ with only the first
$k$ rows. We count the number of ones in this matrix by rows to obtain
$\sum_{i=1}^k \alpha_i^+$ and note that if $j \le k$ there are
$\sum_{i=1}^k a_{ij} = \min(\alpha_j^-, k-1)$ ones in column $j$ and if $j
> k$ there are
$\sum_{i=1}^k a_{ij} = \min(\alpha_j^-, k)$ ones in column $j$, then the
count of ones by column is
$\sum_{j=1}^k\min(\alpha_j^-,k-1) + \sum_{j=k+1}^n\min(\alpha_j^-, k)$. Thus
$\sum_{j=1}^k\min(\alpha_j^-,k-1) + \sum_{j=k+1}^n\min(\alpha_j^-, k) =
\sum_{i=1}^k \alpha_i^+$, as desired. Notice that this implication does not
require positive lexicographic order. 

$(\ref{it:fulk})\Rightarrow (\ref{it:def})$. Assume that $\alpha$ is in
positive lexicographic order and that we have equality in the
Fulkerson-Chen inequalities. We will form the adjacency matrix $A$ one
column at a time. Let $c(i,k) = |\{j\le k \mid a_{ji} = 1\}|$, the
number of ones in the first $k$ rows of the $i^{th}$ column. 
For any $k$, we have that the number of ones in the
submatrix $A_k$ is given by 
$\sum_{i=1}^k \alpha_i^+ = \sum_{i=1}^n c(i,k)$. Notice that for each
$i$ and $k$ we have
\[
c(i,k) \le \left\{
\begin{array}{l@{\qquad}l}
\min(\alpha_i^-, k-1) & i \le k\\
\min(\alpha_i^-, k) & i > k.
\end{array}
\right.
\]
Since we have equality in the Fulkerson-Chen conditions, we must also
have equality for each $c(i,k)$. In particular, considering column
$i$, if $\alpha_i^- \ge i$, then let $k = \alpha_i^-+1$. Notice that
$c(i,k) = \min(\alpha_i^-,k-1) = \alpha_i^-$, and, since $a_{ii} = 0$,
there are only $\alpha_i^-$ positions for the ones in this column of
$A_k$. Therefore, $a_{ji} = 1$ for every $j \not= i$ and $j \le k =
\alpha_i^-+1$. This is the number of ones in this column so the rest are
zeros. If $\alpha_i^- < i$, let $k = \alpha_i^-$. Again, $c(i,k) =
\min(\alpha_i^-,k) = \alpha_i^-$ and there are only $\alpha_i^-$
positions for ones in this column of $A_k$. Thus, $a_{ji} = 1$ for every
$j \le k$ and $a_{ji} = 0$ for every $j > k$. Each of these choices
was forced, so every arc in $G$ is forced and $G$ is the unique
realization of $\alpha$. The only place that this requires positive
lexicographic order is the set-up: to satisfy the Fulkerson-Chen conditions
with equality requires $\alpha$ to be in positive lexicographic order.
\end{proof}

We call any digraph that satisfies these conditions {\bf threshold}.
This definition generalizes the well-studied concept of threshold
graphs~\cite{Mahadev1995}.

As mentioned above, Rao, Jana and Bandyopadhyay~\cite{Rao1996} showed
the equivalence of conditions \ref{thm:td}.\ref{it:def} and
\ref{thm:td}.\ref{it:td} in the context of Markov chains for
generating random zero-one matrices with zero trace. Condition
\ref{thm:td}.\ref{it:fulk} appears in the literature (for example,
Berger~\cite{Berger2011} states this as the definition of threshold
digraphs), but we cannot find a proof of its equivalence to the first
two conditions. Condition \ref{thm:td}.\ref{it:adj} appears to be
entirely new as of this paper, although
Berger~\cite{Berger2011b,Cogis1982} briefly mentions a similar
criteria, without proof, in the context of corrected Ferrers diagrams
in her thesis.

There are two places where the order of $\alpha$ is important. One is
in the statement of condition \ref{thm:td}.\ref{it:fulk}. The second
is in the proof of that condition \ref{thm:td}.\ref{it:td} implies
condition \ref{thm:td}.\ref{it:adj}. However, since condition
\ref{thm:td}.\ref{it:td} does not depend on the order of the vertices,
but on the graph structure, we may characterize threshold digraphs in
the absence of the condition that $\alpha$ is in positive
lexicographic order. In particular, condition
\ref{thm:td}.\ref{it:adj} gives that the digraph is threshold even
when the degree sequence is unordered.

\begin{corollary}
Let $G$ be a digraph and $A= [a_{ij}]$ an adjacency matrix of $G$. Define
$\alpha_i^+ = \sum_{j=1}^n a_{ij}$ and $\alpha_j^-= \sum_{i=1}^n
a_{ij}$. If for every triple of distinct indices $i$, $j$ and $k$ with $i <
j$ and $a_{jk} =1$, it also holds that  $a_{ik} = 1$, then $G$ is a
threshold digraph. \label{cor:adj}
\end{corollary}

\begin{proof}
We show that such a graph cannot have 2-switches or induced directed 3-cycles.
A 2-switch is formed with
four distinct indices, $i$, $j$, 
$k$ and $l$ so that $a_{ij} = a_{kl} = 1$ and $a_{il} = a_{kj} =
0$. Without the loss of generality, suppose that $i < k$. If condition
\ref{thm:td}.\ref{it:adj} holds, then $a_{kl} = 1$ gives $a_{il} = 1$, so there
are no 2-switches. Similarly, an induced directed
3-cycle is formed with three distinct indices, $i$, $j$ and $k$ so
that $a_{ij} = a_{jk} = a_{ki} = 1$ and $a_{ik} = a_{kj} = a_{ai} =
0$. Suppose that $i$ is the smallest of the three indices.  If condition
\ref{thm:td}.\ref{it:adj} holds and $a_{jk} = 1$, then $a_{ik}=1$ so we
cannot have 
an induced directed 3-cycle, either.
\end{proof}

Corollary~\ref{cor:adj} gives us a constructive method for creating
threshold digraphs.

\begin{corollary}\label{cor:construct}
Given a sequence $\beta = (\beta_1,\ldots, \beta_n)$, with $0 \le
\beta_j < n$ for all $j$, if we define an $n\times n$
matrix $A = [a_{ij}]$ by
\[
a_{ij} = \left\{\begin{array}{l@{\qquad\qquad}l}
1 & i<j \ \mathrm{and} \ i \le \beta_j\\
1 & i>j \ \mathrm{and} \ i \le \beta_j+1\\
0 & \mathrm{otherwise},
\end{array}
\right.
\]
then the matrix $A$ is the adjacency matrix of a threshold
digraph. Furthermore, if $G$ is a threshold digraph and $\alpha =
\big((\alpha_1^+,\alpha_1^-),\ldots, (\alpha_n^+,\alpha_n^-)\big)$,
then the sequence $\beta = (\alpha^-_1,\ldots,\alpha_n^-)$ generates
an adjacency matrix of $G$.
\end{corollary}

\begin{proof}
  Since $A$ satisfies condition \ref{thm:td}.\ref{it:adj},
  Corollary~\ref{cor:adj} gives that it is threshold. For a threshold
  digraph $G$, the only matrix which satisfies both condition
  \ref{thm:td}.\ref{it:adj} and the condition $\sum_{i=1}^n a_{ij} =
  \alpha_j^-$ is the matrix formed as above. Thus, $A$ must be the
  adjacency matrix of $G$.
\end{proof}

Since Corollary~\ref{cor:construct} ties together sequences and
threshold digraphs, one application of it is to provide upper and
lower bounds on the number of threshold digraphs for a given
$n$. However, if we permute a sequence, then the resulting threshold
digraph may or may not be isomorphic.  For example, on three vertices
the six orders of the sequence $(2,1,0)$ produce two non-isomorphic
threshold digraphs. The sequences $(2,1,0),\ (1,2,0)$, and $(2,0,1)$
all produce the same digraph with degree sequence
$((1,2),(1,1),(1,0))$ in positive lexicographic order, while the
remaining three sequences produce the threshold digraph with degree
sequence $((2,0),(1,1),(0,2))$ in positive lexicographic order.

\begin{corollary}
Define $TD(n)$ as the number of threshold digraphs on $n$
vertices. Then $\frac{n^n}{n!} \le TD(n) \le n^n$.
\end{corollary}

\section{Digraph Realizability}\label{S: DigraphRealiz}

The idea of condition \ref{thm:td}.\ref{it:fulk} comes from
what are known as the Fulkerson-Chen inequalities for digraph
realizability. Fulkerson studied digraph realizability in the context
of zero-one matrices with zero trace~\cite{Fulkerson1960}. For a given
degree sequence, Fulkerson gave a system of $2^n-1$ inequalities
that are satisfied if and only if the degree sequence is
digraphical. The formulation that we typically use is due to
Chen~\cite{Chen1966}, which reduces the number of inequalities from
$2^n-1$ to $n$ when the degree sequence is in negative lexicographic
order. Our consideration of threshold digraphs gives a new proof of this
result. 

This proof uses a partial order $\preceq$ on integer sequences. In
particular, for sequences $\alpha = (\alpha_1, \ldots, \alpha_n)$ and
$\beta =(\beta_1,\ldots, \beta_n)$ we say $\alpha \preceq \beta$ if
$\sum_{i=1}^k \alpha_i \le \sum_{i=1}^k \beta_i$ for $k = 1, \ldots,
n-1$ and $\sum_{i=1}^n \alpha_i = \sum_{i=1}^n \beta_i$. One important
property of this partial order is that if $\alpha \not= \beta$ and
$\alpha \preceq \beta$, then there is an index $i$ such
that $\alpha_i < \beta_i$ and a first index $j > i$ with $\sum_{k=1}^j
\alpha_k = \sum_{k=1}^j \beta_j$. 

\begin{theorem}\label{thm:real}
Let $\alpha = \big((\alpha_1^+,\alpha_1^-),\ldots,
(\alpha_n^+,\alpha_n^-)\big)$ be a degree sequence in positive
lexicographic order. There is a digraph $G$ which realizes
$\alpha$ if and only if $\sum\alpha_i^+ = \sum\alpha_i^-$ and for
every $k$ with $1 \le k < n$
\[
\sum_{i=1}^k\min(\alpha_i^-,k-1)+\sum_{i=k+1}^n\min(\alpha_i^-,k) \ge
\sum_{i=1}^k\alpha_i^+.
\]
\end{theorem}

\begin{proof}
Suppose that $G$ realizes $\alpha$ with adjacency matrix $A$. Define
\[c(i,k)=|\{j\le k \mid a_{ji} = 1\}|\] as in the proof of
Theorem~\ref{thm:td}, we see that
\[
\sum_{i=1}^k \alpha_i^+ = \sum_{i=1}^n c(i,k) \le
\sum_{i=1}^k\min(\alpha_i^-,k-1)+\sum_{i=k+1}^n\min(\alpha_i^-,k),
\]
as desired.

Suppose that $\alpha$ is a sequence which satisfies the above
inequalities. Construct an adjacency matrix $T$ as in
Corollary~\ref{cor:construct} from the sequence $\alpha^-$. We will
iteratively form a sequence of digraphs $T= B^{(0)}, B^{(1)}, \ldots,
B^{(t_{max})}$ with $B^{(t_{max})}$ an adjacency matrix realizing
$\alpha$, with $\beta^{(t)}$ the sequence of row sums in the matrix
$B^{(t)}$. By hypothesis, $\alpha^+ \preceq \beta^{(0)}$.  If
$\alpha^+=\beta^{(0)}$, then $t_{max} = 0$ and $T = B^{(0)}$ is the
adjacency matrix of the desired graph. Otherwise, define $t_{max} =
\frac12\sum_{i=1}^n|\alpha_i^+-\beta_i^{(0)}|$, and let $r(1,t)$ and
$r(2,t)$ be indices such that $r(1,t)$ is the smallest index where
$\alpha_{r(1,t)}^+ < \beta_{r(1,t)}^{(t)}$ and $r(2,t)$ the first
index after $r(1,t)$ such that $\sum_{i=1}^{r(2,t)} \alpha_i^+ =
\sum_{i=1}^{r(2,t)}\beta_i^{(t)}$.  For $t < t_{max}$, define
$\beta^{(t+1)} = (\beta^{(t+1)}_1, \ldots, \beta^{(t+1)}_n)$ as the
sequence with
\[
\beta^{(t+1)}_i = \left\{\begin{array}{l@{\qquad}l}
\beta^{(t)}_i - 1 & i = r(1,t)\\
\beta^{(t)}_i + 1 & i = r(2,t)\\
\beta^{(t)}_i & \mathrm{otherwise}.
\end{array}\right.
\]
Clearly $\beta^{(t)} \succ \beta^{(t+1)} \succeq \alpha^+$.  Since $\alpha_{r(2,t)}^{+}-1 \geq \beta_{r(2,t)}^{(t)}$ and
$\alpha_{r(1,t)}^{+}+1 \leq \beta_{r(1,t)}^{(t)}$, we have
\[\beta_{r(1,t)}^{(t)} - \beta_{r(2,t)}^{(t)} \ge
(\alpha_{r(1,t)}^{+}+1)-(\alpha_{r(2,t)}^{+}-1) \ge 2.\] Thus, there
are columns $c(1,t)$ and $c(2,t)$ of $B^{(t)}$ that have ones in row
$r(1,t)$ and zeros in row $r(2,t)$.  Either $c(1,t)\neq r(2,t)$ or
$c(2,t)\neq r(2,t)$; therefore, without the loss of generality, we may suppose
that $c(1,t) \not= r(2,t)$. Let $B^{(t+1)}$ be the matrix with
\[
b_{ij}^{(t+1)} =  \left\{\begin{array}{l@{\qquad}l}
0 & i = r(1,t), j = c(1,t) \\
1 & i = r(2,t), j = c(1,t) \\
b_{ij}^{(t)} & \mathrm{otherwise}.
\end{array}\right.
\]
Since $\sum_{i=1}^n|\alpha_i^+-\beta_i^{(t+1)}| =
\sum_{i=1}^n|\alpha_i^+-\beta_i^{(t)}|-2$, we have that
$$\sum_{i=1}^n|\alpha_i^+-\beta_i^{(t_{max})}| =
\sum_{i=1}^n|\alpha_i^+-\beta_i^{(0)}| -2t_{max} = 0.$$
Therefore, 
$\beta^{(t_{max})} = \alpha^+$ and $B^{(t_{max})}$ is a realization of
$\alpha$, as desired.
\end{proof}

This proof is constructive;
given a digraphical degree sequence 
$\alpha$, we can construct a realization of $\alpha$ by repeatedly
moving the ones down in the columns as in the proof of
Theorem~\ref{thm:real}. There are other construction algorithms for
digraphs, most notably that of Kleitman and Wang~\cite{Kleitman1973}.

\section{Applications}\label{S: Applications}

What follows is a quick survey of some consequences of
Theorem~\ref{thm:td}. Some details are omitted since the first two
results are immediate. 

Threshold graphs, in the undirected sense, are closely tied to the
theory of split graphs. An analogous study of {\bf split digraphs} is
given in LaMar~\cite{LaMar2012}. Using the fourth characterization of
threshold digraphs and a result by LaMar, we
have Corollary~\ref{cor:split}.

\begin{corollary}\label{cor:split}
Every threshold digraph is a split digraph.
\end{corollary}

There is also a study of the relationship between different threshold
graphs, as subgraphs of one another, by Merris and
Roby~\cite{Merris2005}. As a consequence of the third characterization
of threshold digraphs, we have Corollary~\ref{cor:addarc}.

\begin{corollary}\label{cor:addarc}
  Given a threshold digraph $G$, if $G$ is nonempty, then there is an
  arc $e$ in $G$ such that $G-e$ is a threshold digraph. If $G$ is not
  complete, then there is an arc $e$ not in $G$ such that $G+e$ is a
  threshold digraph.
\end{corollary}

It has been observed that the ordering required by
Theorem~\ref{thm:real} can be relaxed and still only require the $n$
inequalities stated. Berger~\cite{Berger2011} observed that we need
only require nonincreasing order in the first component. Our theorem
suggests that this can be relaxed even more, but it is not readily
apparent which orders should be considered for graphicality. However,
we can show that nonincreasing order in the first component is
sufficient.

\begin{theorem}
Let $\alpha$ be an integer pair sequence satisfying $\alpha_i^+ \ge
\alpha_{i+1}^+$ for every $1 \le i < n$. If $\sum \alpha_i^+ =
\sum\alpha_i^-$ and
\begin{equation}
\sum_{i=1}^k\min(\alpha_i^-,k-1)+\sum_{i=k+1}^n\min(\alpha_i^-,k) \ge
\sum_{i=1}^k\alpha_i^+\label{eqn:sum conditions}
\end{equation}
for $1 < k < n$, then $\alpha$ is digraphical.
\end{theorem}

\begin{proof}
If $\alpha$ is in positive lexicographic order, then this is true by
Theorem~\ref{thm:real}. Otherwise, let $l$ be an index so that
$\alpha_l^+ = \alpha_{l+1}^+$ and $\alpha_l^- < \alpha_{l+1}^-$. Form
the integer pair sequence $\beta$ from $\alpha$ by exchanging $\alpha_l^-$ and
$\alpha_{l+1}^-$. We show that $\alpha$ satisfies all the inequalities if
and only if $\beta$ satisfies all the inequalities. 

From $\alpha^-$, form the matrix $A$ as in
Corollary~\ref{cor:construct} and let $s_i$ be the row sums in
$A$. From $\beta^-$, form the matrix $B$ and let $s_i'$ be the row
sums in $B$. Notice 
\[
\sum_{i=1}^k\min(\alpha_i^-,k-1)+\sum_{i=k+1}^n\min(\alpha_i^-,k) =
\sum_{i=1}^k s_i
\]
and a similar equality holds for the sums $\sum_{i=1}^ks_i'$.

Notice that $A$ and $B$ differ only in the columns $l$ and
$l+1$. Consider the entries in columns $l$ and $l+1$. We have $a_{i,l}
= b_{i,l+1}$ and $a_{i,l+1} = b_{i,l}$ for every $i \notin\{ l,
l+1\}$; therefore, the row sums are equal except at these two
indices. If $a_{l,l+1} = a_{l+1,l}$, then $s_l = s'_l$ and $s_{l+1} =
s'_{l+1}$; therefore, since $s$ and $s'$ are the same sequence, we
have that $\sum_{i=1}^k s_i \ge \sum_{i=1}^k \alpha_i^+$ if and only
if $\sum_{i=1}^k s_i' \ge \sum_{i=1}^k \alpha_i^+$. In general, we
wish to show that 
$\sum_{i=1}^k s_i \ge \sum_{i=1}^k \alpha_i^+$ for all $k$ if and only if
$\sum_{i=1}^k s_i' \ge \sum_{i=1}^k \alpha_i^+$ for all $k$.

Since $\alpha_l^- < \alpha_{l+1}^-$ it remains only to consider the
case where $a_{l,l+1} = 1$ and $a_{l+1,l} = 0$. In this case, the
construction of $A$ gives that $s_l > s_{l+1}$. We also have that
$s_l' = s_l-1$ and $s_{l+1}' = s_{l+1} + 1$, thus $\sum_{i=1}^k s_i =
\sum_{i=1}^k s_i'$ for every $k \not= l$ and $\sum_{i=1}^l s_i =
\sum_{i=1}^l s_i' +1$. Therefore, for $k < l$ or $k > l+1$, we have
that $\sum_{i=1}^k s_i \ge \sum_{i=1}^k \alpha_i^+$ if and only if
$\sum_{i=1}^k s_i' \ge \sum_{i=1}^k \alpha_i^+$. 

Since the sequences fail the conditions \ref{eqn:sum conditions} with
$k < l$ at the same time, and one failed condition is enough to not
pass this graphicality test, we assume that
$\sum_{i=1}^k s_i = \sum_{i=1}^k s_i' \ge \sum_{i=1}^k \alpha_i^+$ for
$k < l$.  The only way to have exactly one of the conditions
\ref{eqn:sum conditions} fail at $k=l$ is if $\sum_{i=1}^l s_i' <
\sum_{i=1}^l \alpha_i^+$ and $\sum_{i=1}^l s_i \ge \sum_{i=1}^l
\alpha_i^+$. Thus, $\sum_{i=1}^l s_i = \sum_{i=1}^l \alpha_i^+$
and $s_l \le \alpha_l^+$. Both $\alpha$ and $\beta$ fail at least one
condition since
$$\alpha_{l+1}^+ = \alpha_l^+  \ge s_l > s_{l+1}$$ 
implies that $\sum_{i=1}^{l+1} \alpha_i^+ > \sum_{i=1}^{l+1} s_i$.
\end{proof}

This section is only a brief overview of some of the applications of
threshold digraphs. The uses of threshold graphs in various
disciplines has been studied extensively, as shown in Mahadev and
Peled's text~\cite{Mahadev1995}. Our paper is only a starting point of
such a study for threshold digraphs.

\section*{Acknowledgements} We would like to thank Yi-Kai Liu
and Rene Peralta for their comments on our manuscript.

\bibliographystyle{plain}
\bibliography{threshold}

\end{document}